\documentclass[preprint,12pt]{elsarticle}

\usepackage[T1]{fontenc}
\usepackage[utf8]{inputenc}
\usepackage{lmodern,microtype}
\usepackage{amsmath,amssymb,amsthm,mathtools}
\usepackage{enumitem}
\usepackage{xcolor}
\usepackage{hyperref}
\hypersetup{colorlinks=true,linkcolor=blue!45!black,citecolor=blue!45!black,
urlcolor=blue!45!black,
pdftitle={Strict lifts and cohomological obstructions in symplectic Lie--Rinehart--Jacobi prequantization},
pdfauthor={Basile Guy Richard Bossoto, Servais Cyr Gatsé and Olivier Mabiala Mikanou}}
\journal{Journal of Geometry and Physics}

\newtheorem{theorem}{Theorem}[section]
\newtheorem{proposition}[theorem]{Proposition}
\newtheorem{lemma}[theorem]{Lemma}
\newtheorem{corollary}[theorem]{Corollary}
\theoremstyle{definition}
\newtheorem{definition}[theorem]{Definition}
\newtheorem{example}[theorem]{Example}
\theoremstyle{remark}
\newtheorem{remark}[theorem]{Remark}

\newcommand{\G}{\mathcal G}
\newcommand{\Gh}{\widehat{\mathcal G}}
\newcommand{\Alt}{\operatorname{Alt}}
\newcommand{\Diff}{\operatorname{Diff}}
\newcommand{\Der}{\operatorname{Der}}
\newcommand{\Ker}{\operatorname{Ker}}
\newcommand{\id}{\operatorname{id}}
\newcommand{\Lie}{\mathcal L}
\newcommand{\curl}{\operatorname{curl}}
\newcommand{\Obs}{\operatorname{Obs}}
\newcommand{\Cont}{\operatorname{Cont}_{\mathrm{str}}}
\newcommand{\ham}{\mathfrak{ham}}
\newcommand{\hamloc}{\mathfrak{ham}_{\mathrm{loc}}}
\newcommand{\hr}{\widehat\rho}

\begin{document}
\begin{frontmatter}

\title{Strict lifts and cohomological obstructions in symplectic Lie--Rinehart--Jacobi prequantization}

\author[umng]{Basile Guy Richard Bossoto\corref{cor1}}
\ead{basile.bossoto@umng.cg}
\author[umng]{Servais Cyr Gatsé}
\author[umng]{Olivier Mabiala Mikanou}

\cortext[cor1]{Corresponding author.}
\affiliation[umng]{organization={Faculté des Sciences et Techniques, Université Marien Ngouabi},
            city={Brazzaville},
            country={Republic of the Congo}}

\begin{abstract}
Let $(\G,\rho,\omega)$ be a symplectic Lie--Rinehart--Jacobi algebra over a
commutative algebra $A$, equivalently a Lie--Rinehart algebra
$(\G,\bar\rho)$ endowed with a $1$-cocycle $\chi$ and a nondegenerate
$2$-form closed for the twisted differential $d_\rho=d_{\bar\rho}+\chi\wedge$.
Starting from the standard extension by the $2$-cocycle $\omega$, we study
the strict lifting problem for its canonical $1$-form $\eta$ in the LRJ
setting. We prove that the strict symmetries of $\eta$ are precisely the
elements $X_f+fe$ and that $f\mapsto X_f+fe$ identifies the Jacobi Lie
algebra $(A,\{\,,\})$ with this Lie algebra of strict symmetries. For
locally Hamiltonian elements, the obstruction to a strict lift is the
class $[i_x\omega]\in H^1_\rho(\G,A)$, yielding the exact sequence
$0\to\ham\to\hamloc\to H^1_\rho(\G,A)\to0$. We also relate Okassa's
symplectic curl to the variation of $d\eta$ and establish a conformal
triviality criterion. In the smooth case the framework reduces to locally
conformally symplectic geometry, and on a closed surface of genus $g\geq2$
we obtain structures that are not conformally trivial and whose
obstruction space has dimension $2g-2$.
\end{abstract}


\end{frontmatter}


\section{Introduction}

Lie--Rinehart algebras \cite{Rinehart1963} provide the algebraic framework
for differential calculus on Lie algebroids. Replacing derivations of the
base algebra by differential operators of order at most one, Okassa
introduced Lie--Rinehart--Jacobi (LRJ) algebras and developed their
cohomological and symplectic calculus
\cite{Okassa2007,Okassa2008,OkassaBook}. In particular, a closed
nondegenerate $2$-form on an LRJ algebra induces a Jacobi bracket on the
base algebra, and symplectic LRJ structures are closely related to contact
geometry \cite{Okassa2011}. These facts, together with the notions of
Hamiltonian and locally Hamiltonian elements used below, belong to the
established LRJ theory and constitute the background of the present work.

As recalled in Proposition~\ref{prop:decomp}, an LRJ algebra $(\G,\rho)$
may be described as a Lie--Rinehart algebra $(\G,\bar\rho)$ endowed with a
$1$-cocycle $\chi=\rho(\cdot)(1_A)$, with twisted differential
\[
d_\rho=d_{\bar\rho}+\chi\wedge.
\]
This is the algebraic counterpart of a Lie algebroid equipped with a
$1$-cocycle, as in the theory of generalized Lie algebroids or Jacobi
algebroids \cite{IglesiasMarrero2001,GrabowskiMarmo2001}. The use of
$1$-cocycles and conformal transformations in Jacobi geometry is therefore
not new; here it is used to isolate precisely what is specific to the
strict lifting problem for symplectic LRJ algebras.

In the smooth case $A=C^\infty(M)$, $\G=\mathfrak X(M)$, $\bar\rho=\id$, a
symplectic LRJ form is exactly a locally conformally symplectic (LCS) form,
and $H^\bullet_\rho$ is the Lichnerowicz (or Morse--Novikov) cohomology
defined by the closed $1$-form $\chi$ \cite{Vaisman1985,GuediraLichnerowicz1984}
(see Remark~\ref{rem:lcs}). In that setting, Hamiltonian vector fields and
the role of first twisted cohomology in distinguishing locally Hamiltonian
from Hamiltonian fields are already part of the LCS picture
\cite{Vaisman1985,HallerRybicki1999,BelgunGoertschesPetrecca2019}; in
particular, twisted Hamiltonianity is naturally controlled by
Morse--Novikov cohomology. Thus neither the appearance of a first
cohomology obstruction nor its twisted analogue is claimed here as a new
phenomenon. What is specific to the present work is its formulation for an
arbitrary symplectic LRJ algebra and its identification with the obstruction
to preserving the canonical form of the rank-one LRJ extension strictly.

Extensions defined by closed $2$-cocycles and their relation with
prequantization are classical in symplectic geometry and in the
Lie--Rinehart setting \cite{Kostant1970,Souriau1970,Huebschmann1990}.
Geometric quantization and prequantization of Jacobi manifolds were also
studied by de León, Marrero and Padrón \cite{LeonMarreroPadron1997}.
Accordingly, the purpose of this paper is not to introduce the general
idea of prequantization by a cocycle. Rather, starting from the rank-one
LRJ extension
\[
\Gh_\omega=\G\oplus Ae
\]
defined by the symplectic cocycle $\omega$, we determine which elements of
$\G$ admit lifts that preserve the canonical form $\eta$ \emph{exactly},
and we identify the corresponding obstruction in the twisted LRJ
cohomology.

The central identity is
\[
\Lie^{\hr}_{x+fe}\eta
 =\pi^*\bigl(d_\rho f-i_x\omega\bigr).
\]
It shows that a lift $x+fe$ is strict precisely when
$i_x\omega=d_\rho f$, or equivalently $x=X_f$. Consequently the map
\[
\mathcal J:A\longrightarrow\Gh_\omega,\qquad
f\longmapsto X_f+fe,
\]
identifies the Jacobi Lie algebra $(A,\{\,,\})$ with the Lie algebra of
strict symmetries of $\eta$. The projection onto $\G$ then yields a
central extension of the Hamiltonian Lie algebra by $H^0_\rho$.

For a locally Hamiltonian element $x$, the closed $1$-form $i_x\omega$
defines the class
\[
\Obs(x)=[i_x\omega]\in H^1_\rho(\G,A).
\]
Its vanishing is exactly the condition for the existence of a strict lift,
and nondegeneracy of $\omega$ gives the exact sequence
\[
0\longrightarrow\ham\longrightarrow\hamloc
\xrightarrow{\ \Obs\ }H^1_\rho(\G,A)\longrightarrow0.
\]
Thus the novelty is not the cohomological obstruction mechanism by itself:
its untwisted version is classical in symplectic geometry and closely
related twisted mechanisms occur in LCS geometry. The point here is that,
in the general LRJ setting, the same class is shown to be precisely the
obstruction to a strict lift in $\Gh_\omega$, and it fits into a canonical
exact sequence of Lie algebras. This unifies the Jacobi bracket, strict
symmetries and twisted cohomology within a single algebraic construction.

We further show that Okassa's symplectic curl measures, up to sign, the
variation of the curvature $d_{\hr}\eta$; that cohomologous symplectic
$2$-cocycles give isomorphic extensions, while the canonical form retains
additional information; and that the condition
\[
\chi=u^{-1}d_{\bar\rho}u,\qquad u\in A^\times,
\]
reduces the construction by a conformal change to the ordinary
Lie--Rinehart/Poisson situation. An algebraic example with
$\chi=c\,t^{-1}dt$, $c\notin\mathbb Z\cdot1_K$, shows that this conformal
reduction need not occur. Finally, on a closed surface of genus $g\geq2$
endowed with a closed $1$-form $\chi$ that is not exact, we obtain
symplectic LRJ structures that are not conformally trivial, with
$H^0_\rho=0$ and $\hamloc/\ham\cong\mathbb R^{2g-2}$; there both phenomena
studied here occur simultaneously (Example~\ref{ex:surface}).

The contribution of the paper should therefore be understood as the
systematic analysis of strict symmetries and strict lifts in symplectic
LRJ prequantization, together with their precise cohomological obstruction,
rather than as the introduction of cocycle extensions or Jacobi
prequantization themselves.

Section~\ref{sec:prelim} collects the LRJ preliminaries and the Cartan
calculus used throughout. Section~\ref{sec:okassa} recalls the symplectic
calculus needed from Okassa's theory. Section~\ref{sec:conforme} treats
conformal triviality. The rank-one extension and its canonical form are
studied in Section~\ref{sec:contact}. Sections~\ref{sec:strict}
--\ref{sec:curl} contain the strict-lifting and obstruction results, and
Section~\ref{sec:exemples} gives the examples.

\section{Preliminaries}\label{sec:prelim}

Throughout, $K$ is a commutative field of characteristic zero and $A$ is a
commutative unital $K$-algebra. For $a\in A$, $L_a$ denotes multiplication
by $a$. A $K$-linear operator $D:A\to A$ is a \emph{differential operator
of order at most one} if $[[D,L_a],L_b]=0$ for all $a,b\in A$, that is,
\begin{equation}\label{ordre1}
D(ab)=D(a)b+aD(b)-abD(1_A).
\end{equation}
We denote by $\Diff^1_K(A)$ the set of such operators; for
$D\in\Diff^1_K(A)$, the operator $D-L_{D(1_A)}$ is a derivation.

\subsection{Lie--Rinehart--Jacobi algebras}

\begin{definition}[Okassa \cite{Okassa2007}]
A \emph{Lie--Rinehart--Jacobi algebra} over $A$ is an $A$-module $\G$
endowed with a $K$-Lie algebra structure and an $A$-linear map
$\rho:\G\to\Diff^1_K(A)$, the \emph{anchor}, which is a Lie algebra
morphism and satisfies
\begin{equation}\label{LRJ}
[x,ay]=a[x,y]+\bigl(\rho(x)(a)-a\rho(x)(1_A)\bigr)y
\qquad(x,y\in\G,\ a\in A).
\end{equation}
\end{definition}

Set
\[
\chi(x)=\rho(x)(1_A),\qquad \bar\rho(x)=\rho(x)-L_{\chi(x)}.
\]

\begin{proposition}\label{prop:decomp}
Let $(\G,\rho)$ be an LRJ algebra.
\begin{enumerate}[label=(\roman*)]
\item $(\G,\bar\rho)$ is a Lie--Rinehart algebra, and $\chi$ is an
$A$-valued $1$-cocycle of $(\G,\bar\rho)$:
$\bar\rho(x)\chi(y)-\bar\rho(y)\chi(x)-\chi([x,y])=0$.
\item For $x\in\G$ and $a,b\in A$,
\begin{equation}\label{module}
\rho(x)(ab)=\bar\rho(x)(a)\,b+a\,\rho(x)(b).
\end{equation}
In other words, $\rho$ makes $A$ a flat module over the Lie--Rinehart
algebra $(\G,\bar\rho)$.
\end{enumerate}
Conversely, every Lie--Rinehart algebra $(\G,\bar\rho)$ together with a
$1$-cocycle $\chi\in\G^*$ defines an LRJ algebra by $\rho=\bar\rho+L_\chi$.
\end{proposition}

\begin{proof}
By \eqref{ordre1}, $\bar\rho(x)$ is a derivation, and $\bar\rho$ is
$A$-linear. Relation \eqref{LRJ} reads $[x,ay]=a[x,y]+\bar\rho(x)(a)y$.
Since $\rho$ is a Lie algebra morphism,
\[
\chi([x,y])=\rho(x)\chi(y)-\rho(y)\chi(x)
=\bar\rho(x)\chi(y)-\bar\rho(y)\chi(x),
\]
the terms $\chi(x)\chi(y)$ cancelling; this is the cocycle relation.
Expanding $[\bar\rho(x)+L_{\chi(x)},\bar\rho(y)+L_{\chi(y)}]$ gives
$\rho([x,y])=[\bar\rho(x),\bar\rho(y)]+L_{\chi([x,y])}$, hence
$\bar\rho([x,y])=[\bar\rho(x),\bar\rho(y)]$. Finally,
$\rho(x)(ab)=\bar\rho(x)(ab)+\chi(x)ab=\bar\rho(x)(a)b+a\rho(x)(b)$;
flatness follows from the fact that $\rho$ is a Lie algebra morphism. The
converse is checked by the same computations.
\end{proof}

\subsection{The complex and the Cartan calculus}

For $p\geq0$, set $C^p_\rho(\G,A)=\Alt_A^p(\G,A)$ and, for
$\beta\in C^p_\rho(\G,A)$,
\begin{align}\label{drho}
(d_\rho\beta)(x_0,\ldots,x_p)
&=\sum_i(-1)^i\rho(x_i)\,\beta(x_0,\ldots,\widehat{x_i},\ldots,x_p)\notag\\
&\quad+\sum_{i<j}(-1)^{i+j}
\beta([x_i,x_j],x_0,\ldots,\widehat{x_i},\ldots,\widehat{x_j},\ldots,x_p).
\end{align}
In degrees zero and one,
\begin{equation}\label{d01}
(d_\rho a)(x)=\rho(x)(a),\qquad
(d_\rho\beta)(x,y)=\rho(x)\beta(y)-\rho(y)\beta(x)-\beta([x,y]).
\end{equation}
The exterior product is normalized without factorial factors; in
particular, for $\alpha\in C^1$ and $\beta\in C^p$,
$(\alpha\wedge\beta)(x_0,\ldots,x_p)=\sum_i(-1)^i\alpha(x_i)\beta(x_0,\ldots,
\widehat{x_i},\ldots,x_p)$.

\begin{proposition}\label{prop:complexe}
Formula \eqref{drho} defines an $A$-multilinear map, and
$(C^\bullet_\rho(\G,A),d_\rho)$ is a complex. Moreover
\[
d_\rho=d_{\bar\rho}+\chi\wedge(\cdot),\qquad d_\rho1_A=\chi .
\]
The operator $d_{\bar\rho}$ is a graded derivation of the exterior product,
but $d_\rho$ is not: $d_\rho(\alpha\wedge\beta)=d_{\bar\rho}\alpha
\wedge\beta+(-1)^{|\alpha|}\alpha\wedge d_\rho\beta$.
\end{proposition}

\begin{proof}
By Proposition~\ref{prop:decomp}, $d_\rho$ is the
Chevalley--Eilenberg--Rinehart differential of $(\G,\bar\rho)$ with
coefficients in the flat module $(A,\rho)$; it is therefore $A$-multilinear
and squares to zero \cite{Rinehart1963}. The difference $d_\rho-d_{\bar\rho}$
only affects the first sum in \eqref{drho} and equals
$\sum_i(-1)^i\chi(x_i)\beta(\ldots,\widehat{x_i},\ldots)=(\chi\wedge\beta)
(x_0,\ldots,x_p)$. The last formula follows, since $d_{\bar\rho}$ is a
graded derivation.
\end{proof}

We denote by $Z^p_\rho$, $B^p_\rho$ and $H^p_\rho(\G,A)$ the cocycles,
coboundaries and cohomology groups. In particular,
$H^0_\rho=\{a\in A:\rho(x)(a)=0\ \forall x\in\G\}$; if $\chi\neq0$, then
$1_A\notin H^0_\rho$.

For $x\in\G$, let $i_x$ denote the interior product and
$\Lie^\rho_x=i_xd_\rho+d_\rho i_x$ the Lie derivative.

\begin{lemma}[Cartan calculus]\label{lem:cartan}
For $x,y\in\G$ and $\beta\in C^p_\rho(\G,A)$,
\[
(\Lie^\rho_x\beta)(y_1,\ldots,y_p)=\rho(x)\beta(y_1,\ldots,y_p)
-\sum_k\beta(y_1,\ldots,[x,y_k],\ldots,y_p),
\]
and one has
$[\Lie^\rho_x,d_\rho]=0$, $[\Lie^\rho_x,i_y]=i_{[x,y]}$ and
$[\Lie^\rho_x,\Lie^\rho_y]=\Lie^\rho_{[x,y]}$.
\end{lemma}

\begin{proof}
The anchor $\rho$ is a representation of the $K$-Lie algebra $\G$ on the
vector space $A$, and $C^\bullet_\rho(\G,A)$ is a subcomplex of the
Chevalley--Eilenberg complex $\Alt^\bullet_K(\G,A)$ of this representation,
stable under the operators $i_x$. The stated identities are the usual Cartan
identities of that complex, which are checked on the explicit formulas.
\end{proof}

\begin{lemma}\label{lem:pullback}
Let $\varphi:(\G_1,\rho_1)\to(\G_2,\rho_2)$ be a morphism of LRJ
algebras, that is, an $A$-linear map preserving brackets and satisfying
$\rho_2\circ\varphi=\rho_1$. Then
$\varphi^*\circ d_{\rho_2}=d_{\rho_1}\circ\varphi^*$ and
$i_x\varphi^*=\varphi^*i_{\varphi(x)}$ for all $x\in\G_1$.
\end{lemma}

\begin{proof}
Substitute $\varphi(x_i)$ in \eqref{drho}.
\end{proof}

\section{Okassa's symplectic calculus}\label{sec:okassa}

The results of this section are due to Okassa
\cite{Okassa2007,Okassa2008},\cite[Ch.~9]{OkassaBook}; we recall them with
short proofs in order to fix conventions.

\begin{definition}[{\cite[Def.~9.42]{OkassaBook}}]
A \emph{symplectic form} on the LRJ algebra $(\G,\rho)$ is a form
$\omega\in C^2_\rho(\G,A)$ such that $d_\rho\omega=0$ and
$\omega^\flat:\G\to\G^*$, $x\mapsto i_x\omega$, is an isomorphism of
$A$-modules. The triple $(\G,\rho,\omega)$ is then a \emph{symplectic LRJ
algebra}.
\end{definition}

For $a\in A$, the \emph{symplectic gradient} $X_a\in\G$ is defined by
\begin{equation}\label{grad}
i_{X_a}\omega=d_\rho a,\qquad\text{i.e.}\qquad\omega(X_a,y)=\rho(y)(a).
\end{equation}
We set $\{a,b\}=-\omega(X_a,X_b)$.

\begin{proposition}[{\cite[Props.~9.43, 9.44 and Thm.~9.45]{OkassaBook}}]
\label{prop:okassa}
For all $a,b,c\in A$:
\begin{enumerate}[label=(\roman*)]
\item $\Lie^\rho_{X_a}\omega=0$;
\item $\{a,b\}=\rho(X_a)(b)$ and $[X_a,X_b]=X_{\{a,b\}}$;
\item $(A,\{\,,\})$ is a Lie algebra, and
$\{a,bc\}=\{a,b\}c+b\{a,c\}-\{a,1_A\}bc$; in other words $(A,\{\,,\})$ is a
Jacobi algebra.
\end{enumerate}
\end{proposition}

\begin{proof}
(i) $\Lie^\rho_{X_a}\omega=i_{X_a}d_\rho\omega+d_\rho d_\rho a=0$.
(ii) By \eqref{grad}, $-\omega(X_a,X_b)=\omega(X_b,X_a)=\rho(X_a)(b)$. By
Lemma~\ref{lem:cartan} and (i),
$i_{[X_a,X_b]}\omega=\Lie^\rho_{X_a}(d_\rho b)=d_\rho(\rho(X_a)(b))
=d_\rho\{a,b\}$, and we conclude by injectivity of $\omega^\flat$.
(iii) Skew-symmetry is clear. By (ii),
$\{\{a,b\},c\}=\rho([X_a,X_b])(c)=\{a,\{b,c\}\}-\{b,\{a,c\}\}$, which is the
Jacobi identity. The generalized Leibniz rule follows from
\eqref{ordre1} applied to $\rho(X_a)$, with $\rho(X_a)(1_A)=\{a,1_A\}$.
\end{proof}

The definition immediately gives $i_{X_{1_A}}\omega=d_\rho1_A=\chi$.

\begin{proposition}\label{prop:xone}
The following conditions are equivalent: (i) $X_{1_A}=0$; (ii) $\chi=0$;
(iii) $\rho$ takes values in $\Der_K(A)$. They imply that $\{\,,\}$ is a
Poisson bracket. The converse holds if the $X_a$ generate $\G$ as an
$A$-module.
\end{proposition}

\begin{proof}
(i)$\Leftrightarrow$(ii) by injectivity of $\omega^\flat$, and
(ii)$\Leftrightarrow$(iii) by definition of $\chi$. If $\chi=0$, then
$\{a,1_A\}=\chi(X_a)=0$. Conversely, if $\{a,1_A\}=0$ for all $a$, then
$\omega(X_a,X_{1_A})=0$ for all $a$; if the $X_a$ generate $\G$, it follows
that $X_{1_A}=0$.
\end{proof}

\begin{definition}[{\cite[Defs.~9.46 and 9.48]{OkassaBook}}]
An element $x\in\G$ is \emph{locally Hamiltonian} if
$\Lie^\rho_x\omega=0$, and \emph{Hamiltonian} if there exists $a\in A$ with
$x=X_a$. We denote by $\hamloc$ and $\ham$ the corresponding sets. The
\emph{symplectic curl} is the $K$-linear map
$\curl_\omega:\G\to C^2_\rho(\G,A)$, $x\mapsto\Lie^\rho_x\omega$.
\end{definition}

Since $d_\rho\omega=0$, one has $\Lie^\rho_x\omega=d_\rho(i_x\omega)$. Hence
\begin{equation}\label{hamZB}
\hamloc=(\omega^\flat)^{-1}(Z^1_\rho),\qquad
\ham=(\omega^\flat)^{-1}(B^1_\rho).
\end{equation}

\begin{proposition}[{\cite[Props.~9.47 and 9.49]{OkassaBook}}]
\label{prop:hamloc}
\begin{enumerate}[label=(\roman*)]
\item For $x,y\in\hamloc$, $i_{[x,y]}\omega=d_\rho\bigl(\omega(y,x)\bigr)$;
in particular $[\hamloc,\hamloc]\subset\ham$. Thus $\hamloc$ is a Lie
subalgebra of $\G$ and $\ham$ is an ideal of it.
\item For $a\in A$ and $x\in\G$,
$\curl_\omega(ax)=d_{\bar\rho}a\wedge i_x\omega+a\curl_\omega(x)$, where
$d_{\bar\rho}a=d_\rho a-a\,d_\rho1_A$.
\end{enumerate}
\end{proposition}

\begin{proof}
(i) By Lemma~\ref{lem:cartan},
$i_{[x,y]}\omega=\Lie^\rho_xi_y\omega-i_y\Lie^\rho_x\omega
=d_\rho i_xi_y\omega+i_xd_\rho i_y\omega=d_\rho(\omega(y,x))$, since
$d_\rho i_y\omega=\Lie^\rho_y\omega=0$.
(ii) By Proposition~\ref{prop:complexe},
$d_\rho(a\,i_x\omega)=d_{\bar\rho}a\wedge i_x\omega+a\,d_\rho(i_x\omega)$.
\end{proof}

\begin{remark}[The smooth case: locally conformally symplectic manifolds]
\label{rem:lcs}
Let $M$ be a smooth manifold, $K=\mathbb R$, $A=C^\infty(M)$ and
$\G=\mathfrak X(M)$. By Proposition~\ref{prop:decomp}, the LRJ structures
on $\G$ with $\bar\rho=\id$ are exactly those given by
$\rho(X)=X+\chi(X)$ with $\chi\in\Omega^1(M)$ closed; see also
\cite[Props.~11.33--11.34]{OkassaBook}. Then $C^\bullet_\rho(\G,A)=
\Omega^\bullet(M)$ and $d_\rho=d+\chi\wedge$ is the Lichnerowicz
differential, often written $d_\theta=d-\theta\wedge$ with $\theta=-\chi$.
A symplectic form is a nondegenerate $2$-form satisfying
$d\omega=-\chi\wedge\omega$, that is, a locally conformally symplectic form
with Lee form $-\chi$ in the convention $d\omega=\theta\wedge\omega$ of
\cite{Vaisman1985}, and $\{\,,\}$ is the Jacobi bracket associated with it
\cite{GuediraLichnerowicz1984}. By Lemma~\ref{lem:cartan},
$\Lie^\rho_X\beta=\Lie_X\beta+\chi(X)\beta$ on forms, so
\[
\hamloc=\{X:\Lie_X\omega=-\chi(X)\,\omega\},\qquad
\ham=\{X:i_X\omega=df+f\chi\ \text{for some }f\}.
\]
In this setting the description of $\hamloc/\ham$ by the first
Lichnerowicz cohomology is classical
\cite{Vaisman1985,HallerRybicki1999}, and Theorem~\ref{thm:exact} below
extends it to arbitrary LRJ algebras. What is specific to the present
framework is, on the one hand, that $A$ may be an arbitrary commutative
algebra, for which conformal triviality is no longer equivalent to the
exactness of $\chi$ (for instance, on $K[t^{\pm1}]$ the form $t^{-1}dt$ is
conformally trivial but not exact, while on $K[s]$ the form $ds$ is exact
but not conformally trivial; see also Remark~\ref{rem:conforme} and
Example~\ref{ex:jacobi}), and, on the
other hand, the interpretation of $\hamloc/\ham$ as the space of
obstructions to strict lifting in $\Gh_\omega$.
\end{remark}

\section{Conformal change and conformal triviality}\label{sec:conforme}

The following proposition shows that a nonzero cocycle $\chi$ is not
enough to make a structure ``genuinely Jacobi''.

\begin{proposition}\label{prop:conforme}
Let $(\G,\rho,\omega)$ be a symplectic LRJ algebra and assume
$\chi=u^{-1}d_{\bar\rho}u$ for some unit $u\in A^\times$. Then:
\begin{enumerate}[label=(\roman*)]
\item $\rho(x)=L_{u^{-1}}\circ\bar\rho(x)\circ L_u$ for all $x\in\G$;
\item $\beta\mapsto u\beta$ is an isomorphism of complexes
$(C^\bullet_\rho,d_\rho)\to(C^\bullet_{\bar\rho},d_{\bar\rho})$;
\item $\Omega=u\,\omega$ is a symplectic form on the Lie--Rinehart algebra
$(\G,\bar\rho)$, and $X_a=X^\Omega_{ua}$ for all $a\in A$;
\item $a\mapsto ua$ is a Lie algebra isomorphism
$(A,\{\,,\}_\omega)\to(A,\{\,,\}_\Omega)$ from the Jacobi algebra of $\omega$
onto the Poisson algebra of $\Omega$.
\end{enumerate}
\end{proposition}

\begin{proof}
(i) $u^{-1}\bar\rho(x)(uf)=u^{-1}\bar\rho(x)(u)f+\bar\rho(x)(f)
=\chi(x)f+\bar\rho(x)(f)=\rho(x)(f)$.
(ii) Since $d_{\bar\rho}$ is a graded derivation,
$d_{\bar\rho}(u\beta)=d_{\bar\rho}u\wedge\beta+u\,d_{\bar\rho}\beta
=u(\chi\wedge\beta+d_{\bar\rho}\beta)=u\,d_\rho\beta$.
(iii) By (ii), $d_{\bar\rho}\Omega=0$, and $\Omega^\flat=u\,\omega^\flat$ is
bijective. Moreover $i_{X_a}\Omega=u\,d_\rho a=d_{\bar\rho}(ua)$.
(iv) By (iii) and Proposition~\ref{prop:okassa},
$\{ua,ub\}_\Omega=\bar\rho(X_a)(ub)=\bar\rho(X_a)(u)b+u\bar\rho(X_a)(b)
=u\rho(X_a)(b)=u\{a,b\}_\omega$.
\end{proof}

\begin{remark}\label{rem:conforme}
The relevant invariant is thus the class of $\chi$ modulo the logarithmic
 derivatives $u^{-1}d_{\bar\rho}u$ of units of $A$. When $A=R[t,t^{-1}]$,
with $R$ a domain, the units are of the form $\lambda t^n$, where
$\lambda\in R^\times$ and $n\in\mathbb Z$. Hence $t^{-1}dt$ is conformally
trivial. More generally, if $K$ has characteristic zero and
$c\notin\mathbb Z\cdot1_K$, then $c\,t^{-1}dt$ is not the logarithmic
derivative of a unit; consequently it is not conformally trivial (see
Example~\ref{ex:jacobi}).
\end{remark}

\section{Cocyclic contactification}\label{sec:contact}

Let $(\G,\rho,\omega)$ be a symplectic LRJ algebra. Let
$\Gh_\omega=\G\oplus Ae$, where $e$ generates a free summand of rank one.
For $x,y\in\G$ and $f,g\in A$, set
\begin{align}
[x+fe,y+ge]_\omega&=[x,y]+\bigl(\rho(x)(g)-\rho(y)(f)+\omega(x,y)\bigr)e,
\label{hatbracket}\\
\hr(x+fe)&=\rho(x),\qquad \pi(x+fe)=x,\qquad \eta(x+fe)=f.\label{hatrho}
\end{align}

\begin{theorem}\label{thm:extension}
Formulas \eqref{hatbracket}--\eqref{hatrho} make $(\Gh_\omega,\hr)$ an LRJ
algebra over $A$, and
\[
0\longrightarrow Ae\longrightarrow\Gh_\omega\xrightarrow{\ \pi\ }\G
\longrightarrow0
\]
an exact sequence of LRJ algebra morphisms. The ideal $Ae$ is abelian and
\begin{equation}\label{xe}
[x,ae]_\omega=\rho(x)(a)\,e,\qquad\text{in particular}\qquad
[x,e]_\omega=\chi(x)\,e .
\end{equation}
\end{theorem}

\begin{proof}
The bracket is $K$-bilinear and skew-symmetric; $\hr$ is $A$-linear with
values in $\Diff^1_K(A)$.

\emph{Jacobi identity.} Let $u=x+fe$, $v=y+ge$, $w=z+he$ and
$\phi(u,v)=\rho(x)(g)-\rho(y)(f)+\omega(x,y)$. The $\G$-component of the
Jacobiator is that of $x,y,z$, hence zero. The vertical component of
$[[u,v]_\omega,w]_\omega$ is
$\rho([x,y])(h)-\rho(z)\phi(u,v)+\omega([x,y],z)$. In the cyclic sum, the
terms containing $h$ are
\[
\rho([x,y])(h)-\rho(x)\rho(y)(h)+\rho(y)\rho(x)(h)=0,
\]
since $\rho$ is a Lie algebra morphism; by cyclic permutation the same
holds for the terms in $f$ and in $g$. The remaining terms are
\begin{multline*}
-\rho(x)\omega(y,z)-\rho(y)\omega(z,x)-\rho(z)\omega(x,y)\\
+\omega([x,y],z)+\omega([y,z],x)+\omega([z,x],y)
=-(d_\rho\omega)(x,y,z)=0.
\end{multline*}

\emph{Anchor.} $\hr([u,v]_\omega)=\rho([x,y])=[\hr(u),\hr(v)]$.

\emph{Rule \eqref{LRJ}.} For $a\in A$, using \eqref{LRJ} in $\G$ and
\eqref{module},
\begin{align*}
[u,av]_\omega
&=a[x,y]+\bar\rho(x)(a)y\\
&\quad+\bigl(\bar\rho(x)(a)g+a\rho(x)(g)-a\rho(y)(f)+a\,\omega(x,y)\bigr)e\\
&=a[u,v]_\omega+\bigl(\hr(u)(a)-a\,\hr(u)(1_A)\bigr)v .
\end{align*}
Finally, $\pi$ preserves brackets and anchors, $[fe,ge]_\omega=0$, and
\eqref{xe} is the case $f=0$, $y=0$ of \eqref{hatbracket}.
\end{proof}

\begin{remark}
Relation \eqref{xe} shows that $e$ commutes with $\G$ if and only if
$\chi=0$, but that the submodule $Ae$ is never central as soon as $\rho$ is
nonzero. In the situation of Proposition~\ref{prop:conforme}, the element
$u^{-1}e$ satisfies $[x,u^{-1}e]_\omega=\rho(x)(u^{-1})e=0$: the
non-centrality of $e$ is meaningful only up to conformal change.
\end{remark}

\begin{proposition}\label{prop:deta}
$d_{\hr}\eta=-\pi^*\omega$.
\end{proposition}

\begin{proof}
By \eqref{d01}, for $u=x+fe$, $v=y+ge$:
$(d_{\hr}\eta)(u,v)=\rho(x)(g)-\rho(y)(f)-\eta([u,v]_\omega)=-\omega(x,y)$.
\end{proof}

\begin{proposition}[Dependence on the class of $\omega$]\label{prop:classe}
Let $\theta\in C^1_\rho(\G,A)$ and $\omega'=\omega+d_\rho\theta$. The map
\[
\Phi_\theta:\Gh_{\omega'}\longrightarrow\Gh_\omega,\qquad
\Phi_\theta(x+fe)=x+\bigl(f+\theta(x)\bigr)e,
\]
is an isomorphism of LRJ algebras satisfying $\pi\circ\Phi_\theta=\pi$ and
\[
\Phi_\theta^*\eta_\omega=\eta_{\omega'}+\pi^*\theta .
\]
In particular, if $\omega=d_\rho\theta$, then $\Gh_\omega$ is isomorphic to
the split extension $\Gh_0=\G\ltimes A$, and $\eta_\omega$ corresponds to
the form $\eta_0-\pi^*\theta$.
\end{proposition}

\begin{proof}
$\Phi_\theta$ is $A$-linear, bijective with inverse $\Phi_{-\theta}$, and
preserves anchors. The difference between the vertical components of
$\Phi_\theta([u,v]_{\omega'})$ and $[\Phi_\theta u,\Phi_\theta v]_\omega$ is
\[
\omega'(x,y)+\theta([x,y])-\omega(x,y)-\rho(x)\theta(y)+\rho(y)\theta(x)
=(\omega'-\omega-d_\rho\theta)(x,y)=0 .
\]
Finally $\eta_\omega(\Phi_\theta(x+fe))=f+\theta(x)$. For the last
statement, take $\omega=0$ and $\omega'=d_\rho\theta$.
\end{proof}

Thus the isomorphism class of the extension depends only on
$[\omega]\in H^2_\rho(\G,A)$, whereas the form $\eta$ depends on more; in
the exact case one recovers the analogue of the classical model
$(M\times\mathbb R,dt-\theta)$.

\begin{definition}
Let $\mathcal H$ be an LRJ algebra whose underlying $A$-module is finitely
generated projective of constant rank $2m+1$. A form $\eta\in\mathcal H^*$
is a \emph{contact form} if $\eta\wedge(d\eta)^m$ generates the invertible
$A$-module $\Lambda^{2m+1}_A\mathcal H^*$. A \emph{Reeb element} is an
$R\in\mathcal H$ such that $\eta(R)=1$ and $i_Rd\eta=0$.
\end{definition}

\begin{proposition}\label{prop:contactform}
Assume $\G$ is finitely generated projective of constant rank $2m$. Then
$\eta$ is a contact form on $\Gh_\omega$, with
$\eta\wedge(d_{\hr}\eta)^m=(-1)^m\eta\wedge\pi^*(\omega^m)$, and $e$ is its
unique Reeb element.
\end{proposition}

\begin{proof}
For every prime ideal $\mathfrak p$ of $A$, the $A_{\mathfrak p}$-module
$\G_{\mathfrak p}$ is free of rank $2m$ and $\omega_{\mathfrak p}$ is
nondegenerate on it; a nondegenerate alternating form on a finitely
generated free module over a local ring admits a symplectic basis
$(e_1,f_1,\ldots,e_m,f_m)$. In the dual basis,
$\omega^m_{\mathfrak p}=m!\,e^1\wedge f^1\wedge\cdots\wedge e^m\wedge f^m$,
which generates $\Lambda^{2m}\G^*_{\mathfrak p}$ since $m!$ is invertible.
Hence $\omega^m$ generates $\Lambda^{2m}\G^*$. Since $\eta(e)=1$ and
$\eta|_\G=0$, $\eta\wedge\pi^*(\omega^m)$ generates
$\Lambda^{2m+1}\Gh_\omega^*\cong\Lambda^{2m}\G^*\otimes_A(Ae)^*$. The formula
follows from Proposition~\ref{prop:deta}.

We have $\eta(e)=1$ and $i_ed_{\hr}\eta=-i_e\pi^*\omega=0$ since $\pi(e)=0$.
If $R=x+fe$ is a Reeb element, then $\pi^*(i_x\omega)=0$, hence $x=0$ by
injectivity of $\pi^*$ and $\omega^\flat$, and then $f=1$.
\end{proof}

\section{Strict lifts}\label{sec:strict}

The results of this section and the next two use only the bijectivity of
$\omega^\flat$; the hypotheses of Proposition~\ref{prop:contactform} only
serve to justify the ``contact'' terminology.

\begin{definition}
An element $u\in\Gh_\omega$ is a \emph{strict symmetry} of $\eta$ if
$\Lie^{\hr}_u\eta=0$. We denote by $\Cont(\Gh_\omega,\eta)$ the set of these
elements; by Lemma~\ref{lem:cartan} it is a Lie subalgebra of $\Gh_\omega$.
\end{definition}

\begin{lemma}[Lifting formula]\label{lem:lift}
For $u=x+fe\in\Gh_\omega$,
\begin{equation}\label{Leta}
\Lie^{\hr}_{x+fe}\eta=\pi^*\bigl(d_\rho f-i_x\omega\bigr).
\end{equation}
\end{lemma}

\begin{proof}
By Proposition~\ref{prop:deta},
$\Lie^{\hr}_u\eta=i_ud_{\hr}\eta+d_{\hr}(\eta(u))=-\pi^*(i_x\omega)+d_{\hr}f$,
and $d_{\hr}f=\pi^*(d_\rho f)$ by Lemma~\ref{lem:pullback}.
\end{proof}

\begin{theorem}\label{thm:strict}
The element $x+fe$ is a strict symmetry of $\eta$ if and only if $x=X_f$.
Thus $\Cont(\Gh_\omega,\eta)=\{X_f+fe:f\in A\}$.
\end{theorem}

\begin{proof}
Since $\pi$ is surjective, $\pi^*$ is injective; by \eqref{Leta}, the
condition is equivalent to $i_x\omega=d_\rho f$, that is, to $x=X_f$.
\end{proof}

\begin{theorem}\label{thm:jaclift}
The map $\mathcal J:A\to\Gh_\omega$, $\mathcal J(f)=X_f+fe$, is $K$-linear,
injective, and satisfies
\begin{equation}\label{morphism}
[\mathcal J(f),\mathcal J(g)]_\omega=\mathcal J(\{f,g\}).
\end{equation}
It is therefore an isomorphism from the Lie algebra $(A,\{\,,\})$ onto
$\Cont(\Gh_\omega,\eta)$.
\end{theorem}

\begin{proof}
Injectivity is read off the vertical component. By \eqref{hatbracket} and
Proposition~\ref{prop:okassa}, the horizontal component of
$[\mathcal J(f),\mathcal J(g)]_\omega$ is $[X_f,X_g]=X_{\{f,g\}}$ and its
vertical component is
$\rho(X_f)(g)-\rho(X_g)(f)+\omega(X_f,X_g)=\{f,g\}+\{f,g\}-\{f,g\}$.
Surjectivity is Theorem~\ref{thm:strict}.
\end{proof}

\begin{corollary}\label{cor:central}
The projection $\pi$ induces a central extension of Lie algebras
\[
0\longrightarrow H^0_\rho\,e\longrightarrow\Cont(\Gh_\omega,\eta)
\xrightarrow{\ \pi\ }\ham\longrightarrow0,
\]
and, equivalently, $0\to H^0_\rho\to(A,\{\,,\})\to\ham\to0$, where the
second arrow is $a\mapsto X_a$.
\end{corollary}

\begin{proof}
$\pi(\mathcal J(f))=X_f$, whence surjectivity onto $\ham$; the kernel
consists of the $\mathcal J(f)$ with $X_f=0$, that is, $d_\rho f=0$, and
then $\mathcal J(f)=fe$. For $c\in H^0_\rho$ we have $X_c=0$, hence
$\{c,g\}=\rho(X_c)(g)=0$ and $[\mathcal J(c),\mathcal J(g)]_\omega=0$.
\end{proof}

\begin{remark}
For $\chi=0$, $K\cdot1_A\subset H^0_\rho$ and one recovers the
Kostant--Souriau central extension. For $\chi\neq0$, $1_A\notin H^0_\rho$
and $\mathcal J(1_A)=X_{1_A}+e$ is not vertical; in
Example~\ref{ex:jacobi}, $H^0_\rho=0$ and $\pi$ is an isomorphism
$\Cont\to\ham$.
\end{remark}

\section{Cohomological obstruction}\label{sec:obstruction}

\begin{theorem}\label{thm:obstruction}
Let $x\in\G$. There exists $f\in A$ such that $x+fe$ is a strict symmetry
of $\eta$ if and only if $x\in\ham$. When $x\in\hamloc$, this is equivalent
to $[i_x\omega]=0$ in $H^1_\rho(\G,A)$. When they exist, the strict lifts of
$x$ form an affine space over $H^0_\rho$.
\end{theorem}

\begin{proof}
By Lemma~\ref{lem:lift}, $x+fe$ is strict if and only if
$i_x\omega=d_\rho f$. If $f$ and $f'$ both work, then $d_\rho(f'-f)=0$.
\end{proof}

\begin{theorem}\label{thm:exact}
The map $\Obs:\hamloc\to H^1_\rho(\G,A)$, $x\mapsto[i_x\omega]$, is a
surjective morphism of Lie algebras onto the abelian Lie algebra
$H^1_\rho(\G,A)$, with kernel $\ham$. Hence there is an exact sequence of
Lie algebras
\[
0\longrightarrow\ham\longrightarrow\hamloc\xrightarrow{\ \Obs\ }
H^1_\rho(\G,A)\longrightarrow0,
\]
and $\Obs(x)$ is exactly the obstruction to the strict lifting of $x$.
\end{theorem}

\begin{proof}
By \eqref{hamZB}, $\omega^\flat$ induces a bijection $\hamloc\to Z^1_\rho$,
whence surjectivity, and $\Ker\Obs=(\omega^\flat)^{-1}(B^1_\rho)=\ham$. By
Proposition~\ref{prop:hamloc}(i), $\Obs([x,y])=[d_\rho\omega(y,x)]=0$, so
$\Obs$ is a morphism into the abelian Lie algebra $H^1_\rho$. The last
statement is Theorem~\ref{thm:obstruction}.
\end{proof}

\section{Symplectic curl and contact curvature}\label{sec:curl}

\begin{proposition}\label{prop:curlcontact}
For $u=x+fe\in\Gh_\omega$,
\[
\Lie^{\hr}_u(d_{\hr}\eta)=-\pi^*\bigl(\curl_\omega(x)\bigr).
\]
In particular, $u$ preserves $d_{\hr}\eta$ if and only if $x\in\hamloc$, and
this condition does not depend on $f$.
\end{proposition}

\begin{proof}
By Lemmas~\ref{lem:cartan}, \ref{lem:pullback} and \ref{lem:lift},
$\Lie^{\hr}_u d_{\hr}\eta=d_{\hr}\Lie^{\hr}_u\eta
=\pi^*d_\rho(d_\rho f-i_x\omega)=-\pi^*d_\rho(i_x\omega)
=-\pi^*\Lie^\rho_x\omega$. We conclude by injectivity of $\pi^*$.
\end{proof}

\begin{corollary}
If $x\in\hamloc$, every lift $x+fe$ preserves $d_{\hr}\eta$, and it
preserves $\eta$ if and only if $d_\rho f=i_x\omega$.
\end{corollary}

\begin{proposition}\label{prop:conformal}
If $u\in\Gh_\omega$ and $h\in A$ satisfy $\Lie^{\hr}_u\eta=h\eta$, then
$h=0$.
\end{proposition}

\begin{proof}
By \eqref{Leta}, $(\Lie^{\hr}_u\eta)(e)=0$, whereas $(h\eta)(e)=h$.
\end{proof}

\begin{remark}
This rigidity comes from $\hr(e)=0$: the admissible conformal factors are
elements of $A$, that is, ``basic functions''. In the classical model
$(M\times\mathbb R,dt-\theta)$, this corresponds to the fact that contact
vector fields invariant under translation in $t$ are strict.
\end{remark}

\section{Examples}\label{sec:exemples}

In Examples~\ref{ex:affine}--\ref{ex:jacobi}, $A$ is a smooth algebra
whose module $\Der_K(A)$ is free
with basis $(\partial_{x_i})$; we then identify
$C^\bullet_\rho(\Der_KA,A)$ with the module of Kähler forms
$\Omega^\bullet_A$, and $d_\rho$ with $d+\chi\wedge$.

\begin{example}[The affine symplectic case]\label{ex:affine}
Let $A=K[p_1,\ldots,p_m,q_1,\ldots,q_m]$, $\G=\Der_K(A)$, $\rho=\id$ and
$\omega=\sum_idp_i\wedge dq_i$. With our convention,
\[
X_f=\sum_i\Bigl(\frac{\partial f}{\partial q_i}\partial_{p_i}
-\frac{\partial f}{\partial p_i}\partial_{q_i}\Bigr),\qquad
\{p_i,q_j\}=-\delta_{ij},
\]
and $\mathcal J(p_j)=-\partial_{q_j}+p_je$, $\mathcal J(q_j)=\partial_{p_j}+q_je$,
so that $[\mathcal J(p_i),\mathcal J(q_j)]_\omega=-\delta_{ij}e$: one
recovers the Heisenberg algebra. By the Poincaré lemma, $H^1_\rho=0$ and
$H^0_\rho=K$; thus $\hamloc=\ham$ and $\Cont$ is the central extension of
$\ham$ by $Ke$.
\end{example}

\begin{example}[Algebraic torus: nonzero obstruction]\label{ex:tore}
Let $A=K[t^{\pm1},s^{\pm1}]$, $\G=\Der_K(A)$, $\rho=\id$ and
$\omega=(ts)^{-1}dt\wedge ds$, which is closed and nondegenerate. We have
$H^1_\rho(\G,A)=K\,[t^{-1}dt]\oplus K\,[s^{-1}ds]$. The elements
\[
x_1=-s\partial_s,\qquad x_2=t\partial_t
\]
satisfy $i_{x_1}\omega=t^{-1}dt$ and $i_{x_2}\omega=s^{-1}ds$; they are
locally Hamiltonian, $\Obs(x_1)$ and $\Obs(x_2)$ form a basis of
$H^1_\rho$, and no lift $x_k+fe$ preserves $\eta$. Here
$\hamloc/\ham\cong K^2$. One also checks that $[x_1,x_2]=0$, in agreement
with Proposition~\ref{prop:hamloc}(i) since $\omega(x_2,x_1)=-1$ is
constant.
\end{example}

\begin{example}[A conformally trivial case]\label{ex:trivial}
Let $A=K[t^{\pm1},s,p,q]$, $\G=\Der_K(A)$, $\chi=t^{-1}dt$,
$\rho(X)(f)=X(f)+\chi(X)f$ and
$\omega=t^{-1}(dt\wedge ds+dp\wedge dq)$. Then $d_\rho\omega=0$, $\omega$ is
nondegenerate, $[\partial_t,e]_\omega=t^{-1}e$ and $X_{1}=-\partial_s$. But
$\chi=t^{-1}dt$ is the logarithmic derivative of the unit $t$: by
Proposition~\ref{prop:conforme}, $\rho=L_{t^{-1}}\circ\id\circ L_t$, the form
$\Omega=t\omega=dt\wedge ds+dp\wedge dq$ is the standard symplectic form,
$X_1=X^\Omega_t$, and $a\mapsto ta$ identifies the Jacobi algebra of
$\omega$ with the Poisson algebra of $\Omega$. Moreover $t^{-1}e$ commutes
with $\G$.
\end{example}

\begin{example}[A genuinely Jacobi case]\label{ex:jacobi}
Let
\[
A=K[t^{\pm1},s,p,q],\qquad \G=\Der_K(A),
\]
where $K$ is a field of characteristic zero, and choose
$c\in K\setminus(\mathbb Z\cdot1_K)$. Set
\[
\chi=c\,t^{-1}dt,\qquad\rho(X)(f)=X(f)+\chi(X)f .
\]
Since $d\chi=0$, $(\G,\rho)$ is an LRJ algebra. Since $K$ is a field, the
units of $A$ are precisely the elements $\lambda t^n$, with
$\lambda\in K^\times$ and $n\in\mathbb Z$. Their logarithmic derivatives are
\[
(\lambda t^n)^{-1}d(\lambda t^n)=n\,t^{-1}dt.
\]
Therefore $\chi=c\,t^{-1}dt$ is not conformally trivial because
$c\notin\mathbb Z\cdot1_K$ (Remark~\ref{rem:conforme}). Let
$\theta=ds+p\,dq$ and
\[
\omega=d_\rho\theta=c\,t^{-1}dt\wedge ds+c\,p\,t^{-1}dt\wedge dq
+dp\wedge dq .
\]
The Pfaffian of $\omega$ in the basis
$(\partial_t,\partial_s,\partial_p,\partial_q)$ equals $c\,t^{-1}$, which is a
unit: $\omega$ is symplectic. One finds
\begin{gather*}
X_1=-\partial_s,\qquad X_p=-\partial_q,\qquad X_q=\partial_p-q\,\partial_s,\\
X_s=c^{-1}t\,\partial_t-s\,\partial_s-p\,\partial_p,\qquad
X_t=-(1+c^{-1})\,t\,\partial_s,
\end{gather*}
hence, for instance,
\[
\{s,1\}=1,\qquad\{p,q\}=-1,\qquad\{s,q\}=q,\qquad
\{t,s\}=-(1+c^{-1})\,t .
\]
Since $\{s,1\}\neq0$, the bracket is not Poisson, and
$[\mathcal J(s),\mathcal J(1)]_\omega=\mathcal J(1)$ with
$\mathcal J(1)=-\partial_s+e$.

\emph{Cohomology.} We have
$A=K[t^{\pm1}]\otimes_K K[s,p,q]$, and the complex
$(\Omega^\bullet_A,d+\chi\wedge)$ is the tensor product of
\[
\bigl(\Omega^\bullet_{K[t^{\pm1}]},d+c\,t^{-1}dt\wedge\bigr)
\]
with the algebraic de Rham complex of $K[s,p,q]$. In degree zero, the
twisted differential on the Laurent polynomial factor is given by
\[
t^k\longmapsto(k+c)\,t^{k-1}dt,\qquad k\in\mathbb Z.
\]
Since $c\notin\mathbb Z\cdot1_K$, one has $k+c\neq0$ in $K$ for every
$k\in\mathbb Z$. As $K$ is a field, each $k+c$ is invertible. Hence this
map is bijective, and the twisted de Rham complex of $K[t^{\pm1}]$ is
acyclic. The algebraic de Rham complex of $K[s,p,q]$ has cohomology $K$ in
degree zero and vanishes in positive degrees. Therefore, by the Künneth
formula,
\[
H^\bullet_\rho(\G,A)=0.
\]
In particular, $H^0_\rho=H^1_\rho=H^2_\rho=0$. Consequently, every
Hamiltonian element has a unique strict lift, $\pi:\Cont\to\ham$ is an
isomorphism, $\hamloc=\ham$, and $\omega$ is exact, as is also seen
directly from $\omega=d_\rho\theta$. Moreover,
$\Gh_\omega\cong\G\ltimes A$ by Proposition~\ref{prop:classe}, with
$\eta$ corresponding to $\eta_0-\pi^*\theta$.
\end{example}

In Example~\ref{ex:tore} the obstruction is nonzero but $\chi=0$, whereas
in Example~\ref{ex:jacobi} the cocycle $\chi$ is not conformally trivial but
all the cohomology vanishes. The following smooth example shows that both
phenomena can occur simultaneously.

\begin{example}[Surfaces of genus $g\geq2$]\label{ex:surface}
Let $\Sigma$ be a closed connected oriented surface of genus $g\geq2$,
$K=\mathbb R$, $A=C^\infty(\Sigma)$ and $\G=\mathfrak X(\Sigma)$, which is a
finitely generated projective $A$-module of constant rank $2$ by the
Serre--Swan theorem. Choose a closed $1$-form $\chi$ whose de Rham class
$[\chi]\in H^1_{dR}(\Sigma)$ is nonzero, and set
$\rho(X)(f)=X(f)+\chi(X)f$. As in Remark~\ref{rem:lcs}, $(\G,\rho)$ is an
LRJ algebra, $C^\bullet_\rho(\G,A)=\Omega^\bullet(\Sigma)$ and
$d_\rho=d_\chi:=d+\chi\wedge$.

\emph{Symplectic structure.} Let $\omega$ be any area form on $\Sigma$.
Since $\Omega^3(\Sigma)=0$, we have $d_\rho\omega=0$, and $\omega^\flat$ is
bijective because $\omega$ is pointwise nondegenerate. Thus
$(\G,\rho,\omega)$ is a symplectic LRJ algebra, and by
Proposition~\ref{prop:contactform}, $\eta$ is a contact form on the rank-$3$
module $\Gh_\omega$, with Reeb element $e$.

\emph{Absence of conformal triviality.} A unit $u\in A^\times$ is a
nowhere vanishing function, and $u^{-1}d_{\bar\rho}u=u^{-1}du=d\log|u|$ is
exact. Since $[\chi]\neq0$, the cocycle $\chi$ is not conformally trivial.
Moreover $\{\,,\}$ is not a Poisson bracket. Indeed, choose $p\in\Sigma$
with $\chi_p\neq0$ and $\xi\in T^*_p\Sigma$ such that
$\chi_p\bigl((\omega_p^\flat)^{-1}\xi\bigr)\neq0$, and let $a\in A$ satisfy
$a(p)=0$ and $da_p=\xi$. Then $(d_\chi a)_p=\xi$, hence
$X_a(p)=(\omega_p^\flat)^{-1}\xi$, and
$\{a,1_A\}(p)=\chi(X_a)(p)\neq0$.

\emph{Cohomology in degree zero.} Let $f\in A$ with $df+f\chi=0$. On a
contractible open set $U$ write $\chi|_U=dh_U$; then $d(e^{h_U}f)=0$ on
$U$, so $f$ vanishes either identically or nowhere on $U$. Since $\Sigma$
is connected, either $f=0$ or $f$ vanishes nowhere; in the latter case
$\chi=-d\log|f|$ would be exact. Hence $H^0_\rho=0$. The same argument,
applied to $-\chi$, gives $H^0_{-\chi}=0$.

\emph{Cohomology in degree two.} The operator $d_\chi-d=\chi\wedge$ has
order zero, so $(\Omega^\bullet(\Sigma),d_\chi)$ is an elliptic complex with
the same principal symbol as the de Rham complex. Consequently its
cohomology is finite dimensional, Hodge theory applies for any Riemannian
metric, and its Euler characteristic equals that of the de Rham complex:
\begin{equation}\label{euler}
\dim H^0_\rho-\dim H^1_\rho+\dim H^2_\rho=2-2g .
\end{equation}
Equivalently, $H^\bullet_\rho$ is the cohomology of $\Sigma$ with
coefficients in a real local system of rank one, whose Euler characteristic
is that of $\Sigma$. Choose a metric whose area form is $\omega$; every
$2$-form can be written $\varphi\,\omega$ with $\varphi\in A$. By elliptic Hodge theory, every cohomology class has a unique harmonic
representative. In degree two, since there is no outgoing differential,
the harmonic representatives are precisely the $2$-forms orthogonal to
$\operatorname{im}(d_\chi:\Omega^1(\Sigma)\to\Omega^2(\Sigma))$.
Thus $H^2_\rho$ is naturally identified with this orthogonal complement.
For $\alpha\in\Omega^1$ and $f\in A$, Stokes' theorem gives
\[
\int_\Sigma f\,d_\chi\alpha=\int_\Sigma\bigl(f\,d\alpha+f\chi\wedge\alpha\bigr)
=-\int_\Sigma(df-f\chi)\wedge\alpha
=-\int_\Sigma(d_{-\chi}f)\wedge\alpha .
\]
Hence $f\omega$ is orthogonal to $d_\chi\Omega^1$ if and only if
$d_{-\chi}f=0$, that is, $f\in H^0_{-\chi}=0$. Therefore $H^2_\rho=0$, and
\eqref{euler} yields
\[
\dim H^1_\rho(\G,A)=2g-2>0 .
\]

\emph{Consequences.} Since $H^0_\rho=0$, every Hamiltonian element has a
unique strict lift, and by Theorem~\ref{thm:jaclift} and
Corollary~\ref{cor:central} the maps
\[
(A,\{\,,\})\xrightarrow{\ \mathcal J\ }\Cont(\Gh_\omega,\eta)
\xrightarrow{\ \pi\ }\ham
\]
are isomorphisms of Lie algebras. By Theorem~\ref{thm:exact},
\[
\hamloc/\ham\cong H^1_\rho(\G,A)\cong\mathbb R^{2g-2}:
\]
there are locally Hamiltonian vector fields $X$, that is, fields with
$\Lie_X\omega=-\chi(X)\,\omega$, none of whose lifts $X+fe$ preserves
$\eta$, although all of them preserve $d_{\hr}\eta$
(Proposition~\ref{prop:curlcontact}). Explicitly, for every
$d_\chi$-closed $1$-form $\alpha$ that is not $d_\chi$-exact, the element
$X=(\omega^\flat)^{-1}(\alpha)$ is such a field. Finally, $H^2_\rho=0$
implies $\omega=d_\chi\theta$ for some $\theta\in\Omega^1(\Sigma)$, so
$\Gh_\omega\cong\G\ltimes A$ by Proposition~\ref{prop:classe}; as in
Example~\ref{ex:jacobi}, the information distinguishing the strict
symmetries is carried by the canonical form $\eta$ and not by the
isomorphism class of the extension.
\end{example}

\begin{remark}
The genus assumption $g\geq2$ is precisely what ensures a nonzero
obstruction space in this construction. For $g=0$, every
closed $1$-form is exact. For $g=1$ and $[\chi]\neq0$, the same argument
gives $H^0_\rho=H^2_\rho=0$, and \eqref{euler} then forces $H^1_\rho=0$.
If $\chi$ is exact, say $\chi=du/u$ with $u>0$, then $\chi$ is conformally
trivial and, by Proposition~\ref{prop:conforme}(ii),
$H^\bullet_\rho\cong H^\bullet_{dR}(\Sigma)$, so that
$\hamloc/\ham\cong\mathbb R^{2g}$, but the structure then reduces to the
symplectic case.
\end{remark}

\section{Concluding remarks}

\subsection*{Comparison with Okassa's construction}
Okassa shows that the anchors of generalized Lie--Rinehart structures on
$\mathfrak X(M)$ are of the form $X\mapsto X+\alpha(X)$ with $\alpha$ closed
\cite[Props.~11.33--11.34]{OkassaBook}, and that a generalized symplectic
structure on the module $\mathcal D(M)=C^\infty(M)\oplus\mathfrak X(M)$, of
rank $1+\dim M$, characterizes a contact structure on an odd-dimensional
manifold $M$ \cite[Ch.~12]{OkassaBook},\cite{Okassa2011}. That
construction goes from contact to symplectic (rank $2m+2$ when
$\dim M=2m+1$); the contactification studied here goes in the opposite
direction, from symplectic of rank $2m$ to contact of rank $2m+1$.

Beyond the infinitesimal picture developed here, a natural next step is
to study ``integral'' classes $[\omega]\in H^2_\rho(\G,A)$ in settings
where an appropriate notion of line bundle, or rank-one module with
connection, is available. This should make it possible to compare the
strict-lifting obstruction obtained here with a genuinely global
prequantization theory, in the spirit of \cite{Huebschmann1990} and
\cite{LeonMarreroPadron1997}. It would also be useful to investigate the
behaviour of these constructions under conformal changes and, when an
integration theory is available, their relation with Jacobi or contact
groupoids.

\subsection*{Acknowledgements}
The authors pay tribute to the late Professor Eugène Okassa, whose
pioneering work on Lie--Rinehart--Jacobi algebras and their symplectic and
contact geometry provides the mathematical foundation for the present
study. His ideas and scientific legacy have been a continuing source of
inspiration for our work. We respectfully dedicate this article to his
memory.

\end{document}